\documentclass[reqno,a4paper,11pt]{article}
\usepackage{amsmath,amsthm,dsfont,amsfonts,amssymb,fancyhdr}
\usepackage{enumerate}
\usepackage[usenames,dvipsnames]{color}
\usepackage{bbm,soul,supertabular,longtable,verbatim,extarrows}
\usepackage{titlesec}
\usepackage{graphicx}
\usepackage[titletoc,toc,title]{appendix}

\usepackage{tikz}
\usetikzlibrary{arrows,decorations.pathmorphing,backgrounds,positioning,fit,petri}

\usepackage{caption}

\usepackage{float}
\usepackage{BOONDOX-cal}
\usepackage{tikz}
\usepackage{booktabs,multirow,makecell}
\usepackage{authblk}
\usepackage[titletoc]{appendix}
\usetikzlibrary[trees]

\usepackage{lipsum}
\usepackage{mathrsfs}
\usepackage{indentfirst}
\usepackage[colorlinks=true, allcolors=blue]{hyperref}
\usepackage[top=2.4cm,bottom=2.2cm,left=2.6cm,right=2cm]{geometry}

\newtheorem{thm}{Theorem}[section]
\newtheorem{lem}[thm]{Lemma}
\newtheorem{pro}[thm]{Proposition}
\newtheorem{de}[thm]{Definition}

\newtheorem{cor}[thm]{Corollary}
\newtheorem{conj}[thm]{Conjecture}

\newtheorem{problem}[thm]{Problem}

\begin{document}
	\title{\bf On co-edge-regular graphs with 4 distinct eigenvalues}
      \author[a]{Hong-Jun Ge}
    \author[a,b]{Jack H. Koolen\thanks{Corresponding author}}

	\affil[a]{\footnotesize{School of Mathematical Sciences, University of Science and Technology of China, 96 Jinzhai Road, Hefei, 230026, Anhui, PR China}}
	\affil[b]{\footnotesize{CAS Wu Wen-Tsun Key Laboratory of Mathematics, University of Science and Technology of China, 96 Jinzhai Road, Hefei, Anhui, 230026, PR China}}
	
		\maketitle
	\pagestyle{plain}
	
	\newcommand\blfootnote[1]{%
		\begingroup
		\renewcommand\thefootnote{}\footnote{#1}%
		\addtocounter{footnote}{-1}%
		\endgroup}
	\blfootnote{E-mail addresses: {\tt ghj17000225@mail.ustc.edu.cn} (H.-J. Ge), {\tt koolen@ustc.edu.cn} (J.H. Koolen).}

\vspace{-30pt}
\begin{abstract}
Tan et al. conjectured that connected co-edge-regular graphs with four distinct eigenvalues and fixed smallest eigenvalue, when having sufficiently large valency, belong to two different families of graphs.
In this paper we construct two new infinite families of connected co-edge-regular graphs with four distinct eigenvalues and fixed smallest eigenvalue, thereby disproving their conjecture. Moreover, one of these constructions demonstrates that clique-extensions of Latin Square graphs are not determined by their spectrum. 
\end{abstract}
\section{Introduction}
For undefined notions we refer to the next section and \cite{BH11,BrVM2022,god01}.
 
Over the past fifty years, strongly regular graphs have attracted a lot of attention, see \cite{BrVM2022}. Notice that a connected non-complete regular graph $G$ is strongly regular if and only if it has exactly three distinct eigenvalues. 
In particular, Neumaier classified strongly regular graphs with fixed smallest eigenvalue in \cite{N79}.
\begin{thm}[{cf. \cite[Theorem 5.1]{N79}}]\label{N51}
	Let $\lambda\geq 2$ be an integer. Except for finitely many exceptions, any strongly regular graph with smallest eigenvalue $-\lambda$ is a Steiner graph, a Latin Square graph, or a complete multipartite graph.
\end{thm}


In this paper, we study co-edge-regular graphs with four distinct eigenvalues and fixed smallest eigenvalue. We will show that for this class of graphs a result like Theorem \ref{N51} is difficult to obtain.

Notice that clique-extensions of  strongly regular graphs are co-edge-regular graphs with at most four distinct eigenvalues.
Hayat et al.\ \cite{HAYAT19} and Tan et al.\ \cite{TKX20}
gave spectral characterizations of clique-extensions of Square grid graphs and triangular graphs under the assumption that graphs are co-edge-regular.



\begin{thm}[{cf. \cite[Theorem 1.1]{HAYAT19}}]\label{Hayat}
	Let $G$ be a co-edge-regular graph with spectrum
	\[
	\left\{(s(2t+1)-1)^{1}, (st-1)^{2t}, (-1)^{(s-1)(t+1)^2},(-s-1)^{t^2}
	\right\},
	\]
	where $s\geq 2$ and $t\geq1$ are integers. If $t\geq11(s+1)^3(s+2)$, then $G$ is the $s$-clique extension of the $(t+1)\times (t+1)$-grid.
\end{thm}


\begin{thm}[{cf. \cite[Theorem 1]{TKX20}}]\label{Tan}
	Let $G$ be a co-edge-regular graph with spectrum
	\[
	\left\{(2st-3s-1)^{1}, (st-3s-1)^{t-1}, (-1)^{\frac{(s-1)t(t-1)}{2}},(-s-1)^{\frac{t(t-3)}{2}}
	\right\},
	\]
	where $s\geq2$ and $t\geq1$ are integers. If $t\geq48s$, then $G$ is the $s$-clique extension of the triangular graph $T(t)$.
\end{thm}

Theorems \ref{Hayat} and Theorem \ref{Tan} suggest that similar results may hold in the general case.
Based on these results Tan et al. \cite{TKX20} conjectured the following:

\begin{conj}[{cf. \cite[Conjecture 3]{TKX20}}]\label{tan}
	Let $G$ be a connected $k$-regular graph with $n$ vertices and co-edge-regular with parameters $\mu$ having four distinct eigenvalues. 
	Let $t\ge 2$ be an integer. There exists a constant $n_t$ such that,
	if $\theta_{min}(G)\ge -t$, $n\ge n_t$ and $k< n-2-\frac{(t-1)^2}{4}$, then either $G$ is the
	$s$-clique extension of a strongly regular graph for  $2\le s\le t-1$ or $G$
	is a $p\times q$-grid with $p>q\ge 2$.
\end{conj}

However, we will show that Conjecture \ref{tan} is false,
by providing two infinite families of co-edge-regular graphs. This demonstrates that the class of co-edge-regular graphs is much more complicated than previously thought.
The first family is based on results of Haemers and Tonchev \cite{HT96}. The graphs in this family do not have $-1$ as an eigenvalue.

The second family is much more surprising since they have $-1$ as an eigenvalue and cospectral with certain clique-extensions of certain Latin Square graphs.
In order to state the following result we will need to introduce co-edge-regular graphs of  level $t$. For a pair $x,y$ of adjacent vertices in a graph $G$, 
let $\lambda(x,y)$ be the number of common neighbours of $x$ and $y$. We say that a co-edge-regular graph is of {\em level} $t$ if 
$\#\{\lambda(x,y) \mid x,y$ are adjacent vertices $\} = t$.
Note that, in the literature, see for example \cite{Gold06}, the notion of a quasi-strongly regular graph of grade $t$ exists. Such a graph is just the complement of a co-edge-regular graph 
of level $t$. 

\begin{thm}\label{intTLS}
Let $q$ be a prime power. For any positive integer $s\geq 2$, such that $q$ is a factor of $s$, there are infinitely many integers $n$ for which a co-edge-regular graph with level at 
least $3$ exists that is cospectral with the $s$-clique extension of each Latin Square graph $LS_{q+1}(qn)$ and hence has exactly four distinct eigenvalues.
\end{thm}

Note that the $s$-clique extension of any non-complete connected strongly regular graph has exactly four distinct eigenvalues and is co-edge-regular with level $2$. 

Theorem \ref{intTLS} demonstrates that to give a spectral characterization of the clique-extensions of strongly regular graphs with smallest eigenvalue at most $-3$ is impossible, even if you assume that the graphs are co-edge-regular.



This paper is organized as follows. In Section 2, we provide definitions and preliminaries. Section 3 presents a combinatorial characterization of connected co-edge-regular graphs with 
exactly four distinct eigenvalues. In Section 4, we discuss a result from Haemers and Tonchev \cite{HT96} and, as a consequence, construct an infinite family of connected co-edge-
regular graphs of level $2$ with exactly four distinct eigenvalues. Finally, in Section 5, we will construct co-edge-regular  graphs of level $3$ that are cospectral certain clique-extensions of certain Latin 
Square graphs. The existence of this family shows Theorem \ref{intTLS}.

\section{Definitions and Preliminaries}
\subsection{Graphs}
A graph $G$ is an ordered pair $(V(G),E(G))$, where $V(G)$ is a finite set and $\displaystyle E(G)\subseteq \binom{V(G)}{2}$. The set $V(G)$ (resp.\ $E(G)$) is called the vertex set (resp.\ edge set) of $G$, and the cardinality of $V(G)$ (resp.\ $E(G)$) is called the order (resp.\ size) of $G$ and is denoted by $n(G)$ (resp.\ $e(G)$). The adjacency matrix of $G$, denoted by $A(G)$, is a symmetric $(0,1)$-matrix indexed by $V(G)$, such that $(A(G))_{xy}=1$ if and only if $xy$ is an edge in $G$. The eigenvalues of $G$ are the eigenvalues of $A(G)$, and the spectral radius, i.e. the largest eigenvalue, of $G$ is denoted by $\rho(G)$. 
Two graphs are called {\em cospectral} if they have the same spectrum.

If $x$ is a vertex of $G$ then $N_G(x) := \{y \in V(G) \mid x \sim y \}$. The subgraph induced on $N_G(x)$ denoted by $\Delta_G(x)$, is called the local graph at $x$. 
If it is clear what is the graph $G$, we omit $G$ in the notation.
A {\em clique} (resp.\ {\em co-clique}) in a graph $G$ is a set of vertices such that each pair of distinct vertices in it are adjacent (resp.\ non-adjacent).
We sometimes consider a clique (resp.\ a co-clique) of $G$ also as an induced subgraph of $G$. 
 
 Let $G$ be a graph.
 		Let $\pi:=\{V_1,\ldots,V_m\}$ be a partition of vertex set $V(G)$, and, for $1\leq i,j\leq m$ and $u\in V_i$, let $b_{ij}(u)$ be the number of neighbors of $u$ in $V_j$. Then partition $\pi$ is called \emph{equitable} if $b_{ij}(u)$
 		is independent from the concrete choice of $u\in V_i$.
 		In this case, the $m\times m$ matrix $B:=(b_{ij})$ is called the \emph{quotient matrix} of $\pi$. It is well known that for an equitable partition of $G$ with quotient matrix $B$, the eigenvalues of $B$ are also eigenvalues of $A(G)$, see for example \cite[Theorem 9.3.3]{god01}.

\subsection{Edge-regular and co-edge-regular graphs}
 For an edge $xy$ in a graph $G$ we denote by $\lambda(x, y)$, the number of common neighbours of $x$ and $y$ in $G$. For two non-adjacent vertices 
 $x$ and $y$ we denote by $\mu(x, y)$, the number of common neighbours of $x$ and $y$ in $G$.
 
 A regular graph is {\em edge-regular} with parameter $\lambda$ if $\lambda(x,y) = \lambda$ for all edges $xy$ of $G$.
 An edge-regular graph has {\em level} $t$ if $\#\{\mu(x,y) \mid x,y $ are non-adjacent distinct vertices$\} = t$. This notion is also known as a quasi-strongly regular graph with grade $t$, see for example \cite{Gold06}.

 A regular graph is \emph{co-edge-regular} with parameter $\mu$ 
 if any two distinct non-adjacent vertices have exactly $\mu$ common neighbors. We say that a co-edge-regular graph has {\em level} $t$ if 
 $\#\{\lambda(x,y) \mid x,y $ adjacent vertices$\} = t$.
 Note that the complement of an edge-regular graph is co-edge-regular and vice versa and that the complement of an edge-regular graph which has  level $t$ is co-edge-regular and has  level $t$ 
 and vice versa. 

 A graph $G$ is called {\em strongly regular} with parameters $(n, k, \lambda, \mu)$ if it has $n$ vertices, is $k$-regular, is edge-regular with parameter $\lambda$ and co-edge-regular with parameter $\mu$. 
 
 The following result is well-known, see \cite[Lemma 10.2.1]{god01}.
 \begin{lem}\label{srg}
 A connected non-complete regular graph $G$ is strongly regular if and only if it has exactly three distinct eigenvalues. 
 \end{lem}
 
 This means that a connected non-complete regular graph is a co-edge-regular graph of level 1 if and only if it has exactly three distinct eigenvalues.

 Goldberg \cite{Gold06} conjectured a possible extension of a result from Juri\v si\'c et al. \cite{JAKT00} who showed it for distance-regular graphs.
 \begin{conj}[{cf. \cite[Conjecture 1]{Gold06}}]\label{conj1}
 	Let $G$ be a $k$-regular and edge-regular graph with parameter $\lambda$, and let $\theta,\theta'$ be two distinct eigenvalues of $G$, different from $k$. Then
 	$$(\theta+\frac{k}{\lambda+1})(\theta'+\frac{k}{\lambda+1})\geq-\frac{k\lambda(k-\lambda-1)}{(\lambda+1)^2}.$$
 \end{conj}	
 
 However, it is worth mentioning that, for any integer $s \geq 2$ and any integer $t \geq 3$, the complement of the $s$-clique extension of a Latin Square graph $LS_{t}(n)$ 
 (see Subsection \ref{LS} for a definition)
 disproves this conjecture, if $n$ is sufficiently large. Moreover, the complement of our construction in Section \ref{tls} also provides counterexamples to this conjecture.

 \subsection{Designs}
 A pair $(\mathcal{P}, \mathcal{B})$  is a {\em $2$-$(v, k, \lambda)$-design} if 
 the following conditions are satisfied:
 \begin{enumerate}
 \item $|\mathcal{P}| = v$;
 \item For all $B \in \mathcal{B}$, we have $B \subseteq \mathcal{P}$ and $|B|= k$;
 \item For each pair of distinct elements $x, y$ of $\mathcal{P}$ there are exactly $\lambda$ elements $B$ of $\mathcal{B}$ such that both $x$ and $y$ are in $B$.
 \end{enumerate}
 
 We call the elements of $\mathcal{P}$ {\em points} and the elements of $\mathcal{B}$ {\em blocks}. 
  
A  {\em $2$-$(v, k, \lambda)$-design} $\mathcal{D} = (\mathcal{P}, \mathcal{B})$ is called {\em resolvable} if we can partition the block set $\mathcal{B}$ into parts $\mathcal{B_1}, \mathcal{B_2},
\ldots \mathcal {B}_t$ such that $\mathcal{B}_i$ partition $\mathcal{P}$ for all $i = 1,2, \ldots, t$. We call the partition $\{ \mathcal{B_1}, \mathcal{B_2}, \ldots, \mathcal{B}_t\}$ 
a {\em resolution} of $\mathcal{D}$. 

The {\em block graph} of a $2$-$(v, k, 1)$-design $(\mathcal{P}, \mathcal{B})$ is the graph with vertex set $\mathcal{B}$ and two distinct blocks are adjacent if they intersect.
 Note that the block graph of a $2$-$(v, k, 1)$-design is a strongly regular graph with smallest eigenvalue $-k$, see \cite[Section 8.5.4A]{BrVM2022}.

 \subsection{Association schemes}
In this section, we introduce the definition of symmetric association schemes. We will not introduce the Bose-Mesner algebra for them.  
The reader is referred to \cite{BBIT21,god01} for more information.

Let $X$ be a finite set and let $\left\{R_{0}, R_{1}, \ldots, R_{D}\right\}$ be a set of non-empty symmetric binary relations on $X$ which partition $X \times X$. For any $i$ $(0\leq i \leq D)$, define the matrix  $A_i$ with entries $0$ and $1$ such that $(A_{i})_{xy}=1$ if and only if $(x,y)\in R_i$. The pair $\mathfrak{X}=(X,\left\{R_{i}\right\}_{i=0}^{D})$ is called a (\textit{symmetric}) \emph{association scheme with $D$ classes} if the following conditions hold:
\begin{enumerate}
  \item $A_{0}=I_{|X|}$, which is the identity matrix of order $|X|$,
  \item $\sum_{i=0}^{D} A_{i}=J_{|X|}$, the all-ones matrix of order $|X|$,
  \item $A_{i}^{\top}=A_{i}$ for all $i\in \{0,1,\ldots, D\}$, where $A_i^\top$ is the transpose of $A_i$,
  \item $A_{i} A_{j}=\sum_{h=0}^{D} p_{i j}^{h} A_{h}$, where $p_{i j}^{h}$ are nonnegative integers, such that for all $x, y \in$ $X$ with $(x, y) \in R_{h}$ the number of $z \in X$ with $(x,z)\in R_i$ and $(z,y)\in R_j$ equals $p_{i j}^{h}$.
\end{enumerate}
We call $R_0$ the \emph{trivial} relation and the other relations are called \emph{non-trivial}. The integers $p_{i j}^{h}$ are called the \emph{intersection numbers} of $\mathfrak{X}$. 
Any association scheme in this paper is a symmetric association scheme. We also call $\mathfrak{X}$ a \emph{$D$-class association scheme}. The matrices $A_i,$ $i=0, 1, \ldots, D$ 
are called the \emph{relation matrices} of the association scheme $\mathfrak{X}$, and the graph $G_i= (X, R_i)$, with adjacency matrix 
$A_i$, is called the \emph{relation graph} for the relation $R_i$ for $i =0, 1, \dots, D$. Note that $G_i$ is $k_i$-regular where $k_i = p^0_{i i}$, and we say that $k_i$ is the valency of the 
relation $R_i$. Note that any relation graph in an association scheme $\mathfrak{X}$ with $D$ classes has at most $D+1$ distinct eigenvalues.  

Van Dam gave a characterization of edge-regular graphs of level 2 with four distinct eigenvalues in \cite{van99}.
\begin{thm}[{cf. \cite[Theorem 5.1]{van99}}]\label{level2}
	Let $G$ be a connected regular graph that has four distinct eigenvalues. Then $G$ is edge-regular of level 2 if and only if $G$ is one of the classes of a 3-class symmetric association scheme.
\end{thm}

\subsection{Orthogonal arrays}\label{LS}

An {\em orthogonal array} $O$ of order $(n, t)$, denoted by  $OA(n, t)$, is a $t \times n^2$ array such that each entry is an element of $[n]$ and 
$\#\{ (r_i(\ell), r_j(\ell)) \mid \ell =1,2, \ldots, n^2\} = n^2$ for $1 \leq i < j \leq t$, where $r_i$ is the $i$-th row of $O$. 
Note that  from an $OA(n, t)$, we can construct $t-2$ mutually orthogonal Latin Squares of order $n$ for $t \geq 3$, and vice versa. 
A {\em group-divisible orthogonal array} $O$ of order $(n, s, t)$, denoted by $GOA(n, s, t)$ with groups  $G_1, G_2, \ldots, G_t$, each consisting of distinct $s$ rows 
is an $st \times n^2$-array such that $\#\{ (r_i(\ell), r_j(\ell)) \mid \ell =1,2, \ldots, n^2\} = n^2$, whenever $r_i$ belongs to a different group than $r_j$. 
We allow repeated rows in a group.

If there exists an $OA(n, t)$, then, for all $s \geq 2$, there exists a $GOA(n, s, t)$. 
 
The following theorem is the well known MacNeish's Theorem \cite{MN22}. 
	\begin{thm}\label{proj}
		If $n=p_1^{\alpha_1}\cdots p^{\alpha_r}_r$ where $p_1,\ldots,p_r$ are distinct primes. Then there exists an $OA(n,r)$ with
		 $r= \min_ip_i^{\alpha_i}+1$ 
	\end{thm}
	
 Now we will define the Latin Square graphs.
	
	\begin{de}
		 Let  $\mathcal{O}$ be an $OA(n, m)$ with columns $c_1, c_2, \ldots, c_{n^2}$. The Latin Square graph $LS_m(n)$ with respect to $\mathcal{O}$ is the graph on 
		 $V = \{ c_1, c_2, \ldots, c_{n^2}\}$ and for $i \neq j$, $c_i \sim c_j$ if there exists an integer $\ell$ such that $c_i(\ell) = c_j(\ell)$. 
	\end{de}
Each Latin Square graph $LS_m(n)$ is a strongly regular graph with parameters 
$$(n^2,(n-1)m,(m-1)(m-2)+n-2,m(m-1))$$
and has spectrum $Spec(LS_m(n))=\{[(n-1)m]^1,[n-m]^{(n-1)m},[-m]^{(n-1)(n+1-m)}\}$. For more background about Latin Square graphs, we refer to \cite{BrVM2022, god01}

For a positive integer $s$, the \emph{ $s$-clique extension} of a graph $G$ is the graph $\tilde{G}$ obtained from $G$ by replacing each
vertex $x\in V(G)$ by a clique $\tilde{X}$ with $s$ vertices, such that $\tilde{x}\sim\tilde{y}$ (for $\tilde{x}\in\tilde{X}$, $\tilde{y}\in\tilde{Y}$) in $\tilde{G}$ if and only if 
$x\sim y$ in $G$. Note that  $\tilde{G}$ has adjacent matrix $A(\tilde{G})=J_s\otimes(A(G)+I_v)-I_{sv}$, where $I$ is the identity matrix, $J$ is the all-ones matrix, and $v$ is the number of vertices in $G$. 
In particular, if  $ Spec(G)=\{[\theta_0]^{m_0},[\theta_1]^{m_1},\ldots,[\theta_r]^{m_r}\}$, then 
$$Spec(\tilde{G})=\{[s(\theta_0+1)-1]^{m_0},[s(\theta_1+1)-1]^{m_1},\ldots,[s(\theta_r+1)-1]^{m_r},[-1]^{(s-1)v}\}.$$
Note that if $G$ and $H$ have the same spectrum, then the $s$-clique extension of $G$ has the same spectrum, as the  the $s$-clique extension of $H$.
For more background about $s$-clique extensions, we refer to \cite{HAYAT19}.		
	
\subsection{Parallel classes}
Let $q$ be a prime power and $\mathbb{F}_q$ be the finite field of order $q$.
A \emph{plane} $P$ is an affine plane in $\mathbb{F}_q^3$.
Two planes $P_1$ and $P_2$ are {\em parallel} if they are disjoint.
A \emph{parallel class} $\mathcal{S}$ is a set of $q$ mutually parallel planes in $\mathbb{F}^3_q$.
Bose first considered these sets in \cite{bose1942} to construct certain designs. In the next result we show that there are special parallel classes in $\mathbb{F}_q^3$.
The following result may be known, but we could not find a reference for it. 
\begin{thm}\label{cut}
There are $q+1$ parallel classes $\mathcal{S}_1,\ldots,\mathcal{S}_{q+1}$ in $\mathbb{F}^3_q$ satisfying:
\begin{enumerate}
	\item $|P\cap Q|=q$ for $P\in \mathcal{S}_i$ and $Q\in\mathcal{S}_j$ if $i\neq j$,
	\item $|P\cap Q\cap R|=1$ for $P\in \mathcal{S}_h$, $Q\in \mathcal{S}_i$ and $R\in\mathcal{S}_j$, if $1 \leq h <i < j \leq q+1$.
\end{enumerate}
\end{thm}
\begin{proof}
	Let $v_{\infty}:=(0,0,1)$ and $v_x:=(1,x,x(x+1))$ for each $x\in \mathbb{F}_q$. Let $x, y, z$ be three distinct elements of $\mathbb{F}_q$.
Then  	$$\det\begin{pmatrix}0 & 0 & 1 \\ 1 & x & x(x+1)\\ 1 & y & y(y+1)	\end{pmatrix}=y-x$$ and  $$\det\begin{pmatrix} 1 & x & x(x+1)\\ 1 & y & y(y+1)\\ 1 & z & z(z+1)	\end{pmatrix}=(x-y)(y-z)(z-x).$$ This means that any three distinct vectors of $\{ v_x \mid x\in \mathbb{F}_q\} \cup \{\infty\}$ are linear independent.
	Let $P_{a}^b:=\{u\in \mathbb{F}^3_q\mid u^Tv_a =b\}$ for $a\in\mathbb{F}_q\cup\{q+1\}$ and $b\in\mathbb{F}_q$.	
	Then  $\mathcal{S}_a:=\{P_a^b\mid b\in\mathbb{F}_q\}$ is a parallel class as these planes are all orthogonal to $v_a$.
	Let $a_1, a_2, a_3 \in \mathbb{F}_q \cup \{ \infty\}$ be pairwise distinct.
	Then we have 
	$|P_{a_1}^{b_1}\cap P_{a_2}^{b_2}|=q$ for $b_1,b_2\in\mathbb{F}_q$, as $v_{a_1}$ and $v_{a_2}$ are linear independent, and  
	$|P_{a_1}^{b_1}\cap P_{a_2}^{b_2}\cap P_{a_3}^{b_3}|=1$ for $b_1,b_2,b_3\in\mathbb{F}_q$, as $v_{a_1}$, $v_{a_2}$ and $v_{a_3}$ are linear independent.
	
\end{proof}

\section{A combinatorial characterization of co-edge-regular graphs with four distinct eigenvalues}

 In this section we will study connected co-edge-regular graphs with at most four distinct eigenvalues. Note that by Lemma \ref{srg}, we have a characterization of connected regular 
 graphs with 3 distinct eigenvalues. 
 We will generalize this to four distinct eigenvalues. 
 
 Let $G$ be a connected $k$-regular graph on $n$ vertices with exactly four eigenvalues $\theta_0=k > \theta_1> \theta_2 >\theta_3$. 
 	Let $A:=A(G)$ be the adjacent matrix of $G$, $n:=|V(G)|$, and $\ell:=\frac{(k-\theta_1)(k-\theta_2)(k-\theta_3)}{n}$.
 	Then $(A-\theta_1I)(A-\theta_2I)(A-\theta_3I)=\ell J$. It follows that
 	\begin{equation}\label{1}
 		A^3-(\theta_1+\theta_2+\theta_3)A^2+(\theta_1\theta_2+\theta_1\theta_3+\theta_2\theta_3)A-\theta_1\theta_2\theta_3 I=\ell J.
 	\end{equation}

  \begin{de}
  \begin{enumerate}
 \item A co-edge-regular graph with parameter $\mu$ is a {\em $(\mu,\gamma)$-strongly co-edge-regular}, if
 	 every pair of non-adjacent vertices $\{x,y\}$ satisfies  $\sum_{z\in N(x)\cap N(y)}\lambda(x,z)=\gamma$. 

 \item A regular graph is {\em $(\alpha,\beta)$-weakly edge-regular} if every pair of adjacent vertices $\{x,y\}$ satisfies
 	$\alpha \lambda(x,y)=\sum_{z\in N(x)\cap N(y)}\lambda(x,z)+\beta$.
\end{enumerate}
 \end{de} 
 
 Note that a non-complete strongly regular graph with parameters $(n, k, \lambda, \mu)$ is $(\mu, \gamma)$-strongly co-edge-regular with $\gamma = \lambda^2$,
 and $(\alpha, \beta)$-weakly edge-regular for any pair $(\alpha, \beta)$ satisfying $\beta = (\alpha - \lambda)\lambda$.
 
 \begin{thm}\label{4ev}
 		Let $G$ be a connected non-complete $k$-regular and co-edge-regular graph on $n$ vertices with parameter $\mu$.
 		Then $G$ has at most $4$ distinct eigenvalues  if and only if it is $(\alpha,\beta)$-weakly edge-regular and $(\mu,\gamma)$-strongly co-edge-regular for some reals 
		$\alpha, \beta, \gamma$. 
 		Furthermore, if $G$ has exactly $4$ distinct eigenvalues $k > \theta_1 > 
	\theta_2 > \theta_3$, then $\alpha-\mu=\theta_1+\theta_2+\theta_3$, $\mu(\alpha-1)+k-\beta-\gamma=\theta_1\theta_2+\theta_1\theta_3+\theta_2\theta_3$, and $\mu(k-\alpha)+\gamma=\frac{(k-\theta_1)(k-\theta_2)(k-\theta_3)}{n}$.
 \end{thm}	 	
 \begin{proof}
 	Let $G$ be a $k$-regular and co-edge-regular graph with parameter $\mu$. 
	
	As $G$ is non-complete, $G$ has at least three distinct eigenvalues. If $G$ has exactly three distinct eigenvalue, then we already have seen that $G$ is strongly and hence 
	$(\alpha,\beta)$-weakly edge-regular and $(\mu,\gamma)$-strongly co-edge-regular for some reals $\alpha, \beta, \gamma$. Now assume that $G$ has at least four distinct eigenvalues. 

 	Notice that $(A^3)_{xx}=2e(\Delta_x)=k(k-1)-\mu(n-k-1)$, where $e(\Delta_x)$ is the number of edges in $\Delta_x$ the local graph at $x$.
	For a vertex $x$ of $G$, let $N(x)$ be the set of neighbours of $x$ in $G$.
 	
 	Let $x$ and $y$ be two distinct non-adjacent vertices. Then $$(A^3)_{xy}=(k-\mu)\mu+\sum_{z\in N(x)\cap N(y)}\lambda(x,z) \mbox{ and } (A^2)_{xy}=\mu.$$ 
	Similarly, if $u$ and $v$ are two adjacent vertices, then 
 	$$(A^3)_{uv}=\sum_{w\in N(u)\cap N(v)}\lambda(u,w)+(k-1-\lambda(u,v))\mu+k \mbox  { and } (A^2)_{uv}=\lambda(u,v).$$

	Assume that $G$ has exactly four distinct eigenvalues $\theta_0=k >\theta_1> \theta_2 > \theta_3$. 
	 By Equation (\ref{1}), for two distinct non-adjacent vertices  $x$ and $y$, the sum 
 	$$\sum_{z\in N(x)\cap N(y)}\lambda(x,z)=\mu(\theta_1+\theta_2+\theta_3-k+\mu)+\ell$$ does not depend on the choice of $x$ and $y$, 
	where $\ell:=\frac{(k-\theta_1)(k-\theta_2)(k-\theta_3)}{n}.$ Hence, $G$ is strongly co-edge-regular with $\gamma=\mu(\theta_1+\theta_2+\theta_3-k+\mu)+\ell$.
 
 	Also by Equation (\ref{1}), for two adjacent vertices $u$ and $v$, the sum
	$$(\theta_1+\theta_2+\theta_3+\mu)\lambda(u,v)=\sum_{w\in N(u)\cap N(v)}\lambda(u,w)+(k-1)\mu+k+\theta_1\theta_2+\theta_1\theta_3+\theta_2\theta_3 -\ell$$ 
	does not depend on the choice of $u$ and $v$. Hence, $G$ is weakly edge-regular with 
	$\alpha=\theta_1+\theta_2+\theta_3+\mu$ and $\beta=(k-1)\mu+k+\theta_1\theta_2+\theta_1\theta_3+\theta_2\theta_3-\ell$.
 	
 	Now let $G$ be a $(\mu,\gamma)$-strongly co-edge-regular and $(\alpha,\beta)$-weakly edge-regular graph. 
 	Let $f_1:=\mu(\alpha-1)+k-\beta-\gamma$, $f_2:=\alpha-\mu$, $f_3:=k(k-1-\alpha)+\mu\alpha-\gamma-\mu(n-k-1)$, and $f_4:=\mu(k-\alpha)+\gamma$.
 	
	Let $x$ be a vertex of $G$. Then $(A^3)_{xx} = k(k-1)-\mu(n-k-1) = f_3 + kf_2 + f_4= (f_1A+ f_2A^2++f_3I+f_4J)_{xx}$.
	
 	Let $x$ and $y$ be two distinct non-adjacent vertices. Then $$(A^3)_{xy}=(k-\mu)\mu+\sum_{z\in N(x)\cap N(y)}\lambda(x,z) \mbox{ and } (A^2)_{xy}=\mu.$$ 
	As $G$ is a $(\mu,\gamma)$-strongly co-edge-regular graph, we see that 
	$(A^3)_{xy} = (k-\mu)\mu+\gamma= f_4 +f_2\mu= f_4J_{xy} + f_2(A^2)_{xy} = (f_1A+ f_2A^2+f_3I+f_4J)_{xy}$.
	
	Using the fact that $G$ is  $(\alpha,\beta)$-weakly edge-regular, it can be checked in a similar way that 
	$(A^3)_{uv} = (f_1A+ f_2A^2+f_3I+f_4J)_{uv}$ for adjacent vertices $u, v$ of $G$. 
	This means that $A^3=f_1A^2+f_2A+f_3I+f_4J$.
 	This implies that $A^4-(k+f_1)A^3+(kf_1-f_2)A^2+(kf_2-f_3)A+kf_3I=0$, since $AJ=kJ$.
 	Therefore, the degree of the minimal polynomial of $A$ is at most four, and thus  $G$ has at most $4$ distinct eigenvalues. 	
 \end{proof}
 
 Of course the above theorem also gives a characterization of edge-regular graphs with four distinct eigenvalues, by looking at their complements. Note that this characterization generalizes the following result of Van Dam \cite{van99}, as relation graphs from $d$-class association schemes
 have at most  $d + 1$ distinct eigenvalues. 
 
 \begin{thm}[{cf. \cite[Theorem 5.1]{van99}}]
 Let $G$ be an edge-regular graph of level $2$. Then $G$ is the relation graph of a 3-class association scheme if and only if it has exactly four distinct four eigenvalues. 
 \end{thm}
 
 
 \section{On a theme of Haemers and Tonchev}
  In strongly regular graphs, the cliques and co-cliques with the property that every vertex outside is adjacent with the same number of vertices inside are characterized by the fact that 
  they satisfy with the equality so-called \emph{Hoffman bound}:
 \begin{thm}[{cf. \cite[Proposition 1.1.7]{BrVM2022}}]\label{hoff}
 	Let $G$ be a strongly regular graph with parameters $(n,k,\lambda,\mu)$ and smallest eigenvalue $-m$.
 	\begin{enumerate} 
 		\item If $C$ is a clique of $G$, then $|C|\leq\frac{m+k}{m}$, with equality if and only if every vertex $x\notin C$ has the same number $\frac{\mu}{m}$ of neighbors in $C$.
 		\item 	If $D$ is a co-clique of $G$, then $|D|\leq \frac{m}{m+k}n$, with equality if and only if every vertex $x\notin D$ has the same number $m$ of neighbors in $D$.

 		\item If a clique $C$ and a co-clique $D$  both meet the bounds of {\rm(i)} and {\rm(ii)}, then $|C\cap D|=1$.
 	\end{enumerate}
 \end{thm}
 In 2021, Haemers wrote a note \cite{HAEMERS2021215} clarifying the history of the Hoffman bound. 
 We call a (co-)clique that meets the Hoffman bound a \emph{Hoffman (co-)clique}. Note that in the literature a Hoffman clique is also called a {\em Delsarte clique}.
 
  \begin{de}
  \begin{enumerate}
 	\item A Hoffman coloring in a strongly regular graph $G$ is a partition of the vertex set into Hoffman co-cliques;
	\item A Hoffman spread in a strongly regular graph $G$ is a partition of the vertex set into Hoffman cliques. 
	\end{enumerate}
 \end{de}
 Note that a Hoffman coloring in a strongly regular graph $G$ is the same as a Hoffman spread in its complement $\overline{G}$.
 
 In \cite{HT96}, Haemers and Tonchev  showed the following result, generalizing a result of Brouwer. 
 \begin{thm}[{cf. \cite[Proposition 4.1]{HT96}}]\label{ht}
 Let $G$ be a connected non-complete strongly regular graph with $n$ vertices and spectrum $Spec(G)=\{[k]^1,[r]^f,[s]^g\}$ having a Hoffman spread $S= \{C_1, C_2, \ldots, C_t\}$.
 Let $E$ be the edge set of $G$ and $F_i$ be the edge set of $C_i$ for $i =1, 2, \ldots, t$. 
 Let $G'$ be the graph with the same vertex set as $G$ and edge set $E \setminus \bigcup_{i=1}^t F_i$. 
 Then $G'$ is the relation graph of a 3-class association scheme. In particular, $G'$ is regular and edge-regular with spectrum
 $$Spec(G')=\{[k+\frac{n-k-1}{r+1}]^1,[r-1]^{m_1},[s-1]^{m_2},[s+\frac{n-k-1}{r+1}]^{m_3}\},$$
 where $m_1=f-\frac{sn}{s-k}+1$, $m_2=g$, and $m_3=\frac{sn}{s-k}-1$. 
 \end{thm}
 Note that the complement $\overline{G_1}$ of $G_1$ of Theorem \ref{ht}, is co-edge-regular and has at most four distinct eigenvalues.
 Therefore, each strongly regular graph with a Hoffman spread can construct a counterexample of Conjecture \ref{tan}.
 There are many examples of strongly regular graphs with Hoffman spreads, see \cite{HT96,TP94a}.
 
By using the resolvable $2$-$(v,t,1)$-design, we give an infinite family of counterexamples of Conjecture \ref{tan} by the above theorem.
 
 Let $\mathcal{D} = (\mathcal{P}, \mathcal{B})$ be a resolvable $2$-$(v, t, 1)$-design, say, with resolution $\rho= \{ \mathcal{B_1}, \mathcal{B_2}, \ldots, \mathcal{B}_{\ell}\}$. 
 This means that the block graph $G$ of $\mathcal{D}$ is a strongly regular graph with parameters $(n,k,\lambda,\mu)$ having smallest eigenvalue $-t$, see \cite[Section 8.5.4A]{BrVM2022},  where $n=\frac{v(v-1)}{t(t-1)}$, $k=\frac{t(v-t)}{t-1}$, $\lambda=t(t-2)+\frac{v-t}{t-1}$, and $\mu=t^2$.
 Note that each $\mathcal{B}_i$ is a co-clique of $G$ with order $\frac{v}{t}$.
 By Theorem \ref{hoff}, each $\mathcal{B}_i$ denotes a Hoffman co-clique of $G$.
 Hence, the resolution $\rho$ gives a Hoffman coloring of $G$.
   
 Now construct the $H$ with vertex set $\mathcal{B}$ such that, for two distinct blocks $B_1, B_2$, we have $B_1 \sim B_2$ if they were adjacent in $G$ or 
 there exists $i$ such that $B_1, B_2 \in \mathcal{B}_i$. Then, by Theorem \ref{ht}, we see that $H$ is a co-edge-regular graph with smallest eigenvalue $-t-1$ having exactly four 
 distinct eigenvalues. 
 
 \begin{thm}
 	The graph $H$ is 
 	regular with valency $\frac{v-t}{t-1}(t+\frac{t-1}{t})$
   and co-edge-regular of level $2$ with parameter $t(t+2)$. Moreover, the spectrum of H is
 	$$Spec(H)=\{[\frac{v-t}{t-1}(t+\frac{t-1}{t})]^1,[\frac{v-t}{t}-t]^{f_1},[\frac{v-1}{t-1}-t-2]^{f_2},[-t-1]^{f_3}\},$$ where  $f_1=\frac{v-t}{t-1}$, $f_2=v-1$, and $f_3=\frac{v-1}{t}(\frac{v-t}{t-1}-t)$.
 \end{thm}
 \begin{proof}
 	By Theorem \ref{level2} and Theorem \ref{ht}, 
 it suffices to show that $H$ is co-edge-regular with parameter $t(t+2)$.
 	Let $x,y\in V(H)$ be two non-adjacent  vertices. Without loss of generality, assume that $x\in \mathcal{B}_1$ and $y\in\mathcal{B}_2$. 
	 	Since $\mathcal{B}_i$ is a Hoffman co-clique of $G$ for each $i$, Theorem \ref{hoff} implies that $x$ (resp. $y$) has exactly $t$ neighbors in $\mathcal{B}_2$ (resp. $\mathcal{B}_1$). Note that $G$ is co-edge-regular with parameter $t^2$.
	 	Therefore, $|N_H(x)\cap N_H(y)|=|N_G(x)\cap N_G(y)|+|N_H(x)\cap \mathcal{B}_2|+|N_H(y)\cap \mathcal{B}_1|=t(t+2)$. This completes the proof.
 \end{proof}

 Note that for fixed $t \geq 2$, there are infinitely many $v$ such that there exists a resolvable $2$-$(v, t, 1)$-design, see for example \cite{RWR1973}.
 These give infinitely many counterexamples to Conjecture \ref{tan}.

 \section{Another construction}\label{tls}
 In this section, we construct a new family of co-edge-regular graphs with four distinct eigenvalues of which one is $-1$. 
 We will use group-divisible orthogonal arrays for this construction. We will call these new graphs the twisted Latin Square graphs.
 In the rest of this section,
 let $q$ be a prime power and $n$ be a positive integer such that there exists a $GOA(n, q, q+1)$, say $\mathcal{O}$ with groups 
 $G_1, G_2, \ldots, G_{q+1}$. Note that, for a given prime power $q$, there are infinitely many $n$ such that exists a $GOA(n, q, q+1)$, by MacNeish's Theorem \cite{MN22}.
For $i\in[q+1]$, order the rows in $G_i$ as $r^i_j$ where $j \in [q]$, since there are exactly $q$ rows in $G_i$.
 
 By Theorem \ref{cut}, there are  $q+1$ parallel classes $\mathcal{S}_1,\ldots,\mathcal{S}_{q+1}$ in $\mathbb{F}^3_q$ satisfying:
\begin{enumerate}
	\item $|P\cap Q|=q$ for $P\in \mathcal{S}_i$ and $Q\in\mathcal{S}_j$ if $i\neq j$,
	\item $|P\cap Q\cap R|=1$ for $P\in \mathcal{S}_h$, $Q\in \mathcal{S}_i$ and $R\in\mathcal{S}_j$, if $1 \leq h <i < j \leq q+1$.
\end{enumerate}

Let $\mathcal{S}_i = \{ P^i_1, P^i_2, \ldots, P^i_q\}$ for $i\in[q+1].$
We consider the planes $P^i_j$ as subsets of order $q^2$ of $\mathbb{F}^3_q$ for $i\in[q+1]$ and $j\in[q]$. 

Now we are ready to define the twisted Latin Square graphs. As we will see later they have the same spectra as certain clique-extensions of certain Latin Square graphs.

The {\em twisted Latin Square graph} with parameters $(q, n)$ with respect to $\mathcal{O}$, abbreviated with $TLS(q, n)$, is the graph with vertex set  $\mathbb{F}^3_q \times [n^2]$.
Two distinct vertices $(\mathbf{x}, i)$ and $(\mathbf{y}, j)$ of $TLS(q, n)$ are adjacent if $i=j$ or if there exists a plane $P^p_{\ell}$ such that $\mathbf x$ and $\mathbf y$ both lie on and
$r^p_{\ell}(i) = r^p_{\ell}(j)$. 

For $s \in [q+1]$, $t \in [q]$ and $\ell \in [n]$ define the set $C(s,t, \ell)$ as the set consisting of the vertices $(\mathbf{x}, i)$ such that $\mathbf x \in P^s_t$ and 
$r^s_{t}(i)= \ell$. It is clear that the set $C(s,t, \ell)$ induces a clique in $TLS(q, n)$ of order $q^2n$. 

\subsection{Structure of $TLS(q,n)$}

We denote $F(i):=\{(\mathbf{x},i)\mid \mathbf{x}\in\mathbb{F}^3_q\}$ for $i =1, 2, \ldots, n^2$
and $P^s_t(i):=\{(\mathbf{x},i)\mid \mathbf{x}\in P^s_t\}$ for $s \in [q+1]$, $t \in [q]$, and $i \in [n^2]$.

\begin{lem}\label{inter}
Let $s_1, s_2 \in [q+1]$, $t_1, t_2 \in [q]$ and $\ell_1, \ell_2 \in [n]$. Assume $C(s_1, t_1, \ell_1) \neq C(s_2, t_2, \ell_2)$. 
If $s_1 \neq s_2$,
then there exists an integer $i$ such that $C(s_1, t_1, \ell_1) \cap C(s_2, t_2, \ell_2)=P^{s_1}_{t_1}(i)\cap P^{s_2}_{t_2}(i)\subseteq F(i)$. 
and $|C(s_1, t_1, \ell_1) \cap C(s_2, t_2, \ell_2)| = q$. If $s_1 = s_2$, then $|C(s_1, t_1, \ell_1) \cap C(s_2, t_2, \ell_2)|=0$.
\end{lem}
\begin{proof}
Assume $s_1 \neq s_2$. 
By Theorem \ref{cut}, the planes $P^{s_1}_{t_1}$ and $P^{s_2}_{t_2}$ intersect in $q$ vectors $\mathbf{x}_1$, $\mathbf{x}_2$,$\ldots$, $\mathbf{x}_q$. 
There exists exactly one position $i\in [n^2]$ such that $r^{s_1}_{t_1}(i)= \ell_1$ and $r^{s_2}_{t_2}(i)= \ell_2$, by the definition of $\mathcal{O}$. 
This shows that $C(s_1, t_1, \ell_1) \cap C(s_2, t_2, \ell_2)=P^{s_1}_{t_1}(i)\cap P^{s_2}_{t_2}(i)$ and $|C(s_1, t_1, \ell_1) \cap C(s_2, t_2, \ell_2)| = q$ in this case. 
Now assume $s_1= s_2$. If $t_1 \neq t_2$, then the planes $P^{s_1}_{t_1}$ and $P^{s_2}_{t_2}$ are parallel and disjoint, and thus 
$|C(s_1, t_1, \ell_1) \cap C(s_2, t_2, \ell_2)| = 0$.
If $t_1 = t_2$, then $\ell_1 \neq \ell_2$ and again it is clear that $|C(s_1, t_1, \ell_1) \cap C(s_2, t_2, \ell_2)| = 0$.
This shows the lemma.
\end{proof}

\begin{lem}\label{reg}
Any twisted Latin Square graph $TLS(q, n)$ is regular with valency $k = q^3 -1 + (q+1)(n-1)q^2$.
\end{lem}
\begin{proof}
Let $(\mathbf{x}, i)$ be a vertex of $TLS(q,n)$. 
Then it has $q^3-1$ neighbours in $F(i)$.
The vector $\mathbf{x}$ lies in the unique plane, say $P^s_{t_s}$, in the parallel class $\mathcal{S}_s$ for $s \in [q+1]$.
So $(\mathbf{x}, i)$ has exactly $q^2(n-1)$ neighbours in $C(s, t_s, r_{t_s}^s(i))-F(i)=C(s, t_s, r_{t_s}^s(i))-P^s_{t_s}(i)$. 
If $s_1 \neq s_2$, then, by 
Lemma \ref{inter}, $C(s_1, t_{s_1}, \ell_1) \cap C(s_2, t_{s_2}, \ell_2)=P^{s_1}_{t_{s_1}}(i)\cap P^{s_2}_{t_{s_2}}(i)\subset F(i)$, as $(\mathbf{x}, i)$ is one of them. 
It follows that the valency of 
 $(\mathbf{x}, i)$ is exactly $q^3-1 + (q+1)(n-1)q^2$. 
\end{proof}

 \begin{lem}\label{xc}
 			Let $s \in [q+1]$, $t\in [q]$ and $\ell \in [n]$.
 			Let $v=(\mathbf{x},i)$ be  a vertex not in $C(s,t,\ell)$.
			Assume that $\mathbf{x}$ lies in planes $P^j_{t_j}$ for $j \in [q+1]$. 
			If $r^{s}_{t}(i)=\ell$, then  $C(s,t,\ell)\cap N(v)=P^s_t(i)$ holds. 
			If $r^{s}_{t}(i) \neq \ell$, then there exists pairwise distinct $m_p$ with $p \in [q+1]-\{s\}$
			such that $$C(s,t,\ell)\cap N(v) = \bigcup_{p=1, p \neq s}^{q+1}P^p_{t_p}(m_p)\cap P^s_t(m_p).$$
			In particular, $v$ has exactly $q^2$ neighbours in $C(s,t,\ell)$.
 \end{lem}
 		
\begin{proof}
 			Assume that $v=(\mathbf{x},i)$.
 			The vector $\mathbf{x}$ lies in a unique plane, say $P^j_{t_j}$, in the parallel class $\mathcal{S}_j$ for $j \in [q+1]$.
 			Hence, $(\mathbf{x},i)$ lies in the clique $C(j,t_j,r^j_{t_j}(i))$ for $j\in[q+1]$.
 			
 			If $r^{s}_{t}(i)=\ell$, then $P^s_t(i)=C(s,t,\ell)\cap F(i)\subset C(s,t,\ell)\cap N(v)$. 
			So $v$ has at least $q^2$ neighbours in $C(s,t,\ell)$, and $t_s \neq t$ as otherwise $v \in C(s,t,\ell)$.
			As 	$(\mathbf{x},i)$ lies in the clique $C(j,t_j,r^j_{t_j}(i))$ for $j\in[q+1]$	and,	by Lemma \ref{inter}, $C(s,t,\ell)\cap C(j,t_j,r_{t_j}^j(i))=P^s_t(i)\cap P^j_{t_j}\subset P^s_t(i)$ if $j\neq s$.
 			This means that  $|C(s,t,\ell)\cap N(v)|=|P^s_t(i)|=q^2$ holds.
 			
 			If $r^{s}_{t}(i)\neq \ell$, then  $C(s,t,\ell)\cap N(v)=C(s,t,\ell)\cap(\bigcup_{i\neq s}C(j,t_j,r^j_{t_j}(i)))$.
 			By Lemma \ref{inter}, there exists an integer $m_j$ such that 
 			$C(s,t,\ell)\cap C(j,t_j,r^j_{t_j}(i))=P^s_t(m_j)\cap P^j_{t_j}(m_j)$ if $j\neq s$.
 			As $r^{s}_{t}(i)\neq \ell$ and $r^{s}_{t}(m_j)= \ell$, we have $i\neq m_j$ for all $j \neq s$.
 			
 			Now we show that $m_j\neq m_{j'}$ for distinct $j,j'\in[q+1]-\{s\}$.
		In fact, we have $(r^j_{t_j}(m_j),r^{j’}_{t_{j’}}(m_j))\neq(r^j_{t_j}(i),r^{j'}_{t_{j'}}(i))$, by the definition of $\mathcal{O}$, as $i\neq m_j$.
 			Thus, $m_j\neq m_{j'}$ for distinct $j,j'\in[q+1]-\{s\}$.
 			
 			This implies that $|C(s,t,\ell)\cap N(v)|=|C(s,t,\ell)\cap(\bigcup_{i\neq s}C(j,t_j,r^j_{t_j}(i)))|=|\bigcup_{i\neq s}
 			(P^s_t(m_j)\cap P^j_{t_j}(m_j))|=\sum_{i\neq s}|(P^s_t(m_j)\cap P^j_{t_j}(m_j))|=q^2$ by Theorem \ref{cut}.
 		\end{proof}

 		\begin{lem}\label{ev}
 			Each twisted Latin Square graph $TSL(q,n)$ has $q^2(n-1)-1$ as an eigenvalue.
 		\end{lem}
 		\begin{proof}
 			Let $V$ be the vertex set of  $TSL(q,n)$. By Lemma \ref{reg}, $TLS(q, n)$ is regular with valency $k = q^3 -1 + (q+1)(n-1)q^2$.
 			By Lemma \ref{xc}, for $s \in [q+1]$, we have a equitable partition $\{C(s,t,\ell),V-C(s,t,\ell)\}$ with quotient matrix 
			$\footnotesize\begin{pmatrix}nq^2-1 & k-nq^2+1 \\ q^2 & k-q^2 	\end{pmatrix}$  having eigenvalues  $q^2(n-1)-1$ and $k$. 
			Therefore, $q^2(n-1)-1$ is an eigenvalue of $TSL(q,s)$.
 		\end{proof}
 		
                  \begin{lem}\label{plane} Let  $u= (\mathbf{x}, i)$ be a vertex in $TLS(q,n)$ where $\mathbf{x} \in \mathbb{F}^3_q$ and $i \in [n^2]$.
                  If $u$ has neighbours in $F(m)$ where $m \neq i$ and $m\in[n^2]$. 
                  Then there exists $s \in [q+1]$ and $t \in [q]$ such that $F(m) \cap N(u) = P^s_t(m)$.
                  \end{lem}
                  \begin{proof}
                  Let $v= (\mathbf{y}, m)$ be a neighbour of $u$ in $F(m).$ This means that there exists $s\in [q+1], t\in [q]$ and $\ell \in [n]$ such that 
                  $u, v \in C(s, t, \ell)$, which implies $\ell = r^s_t(i) = r^s_t(m)$. It follows  that $N(u) \cap F(m) \supseteq P^s_t(m)$. 
                  Let $s_1 \in [q+1], t_1\in [q]$ such that $s_1 \neq s$. Then $r^{s_1}_{t_1}(i) \neq r^{s_1}_{t_1}(m)$ by the definition of $\mathcal{O}$. 
                  As a consequence, we find that $N(u) \cap F(m)=P^s_t(m)$. This shows the lemma.
                  \end{proof}

 		\subsection{Local graphs of $TLS(q,n)$}
 		Let $u=(\mathbf{x},m)$ be a vertex of the graph $TLS(q,n)$, where $\mathbf{x} \in\mathbb{F}^3_q$ and $ m \in [n^2]$.
 		The vector $\mathbf{x}$ lies in the unique plane, say $P^s_{t_s}$, in the parallel class $\mathcal{S}_s$ for $s \in [q+1]$.
 		This means that  $(\mathbf{x},m)$ lies in the clique $C(s, t_s ,r^i_{t_s}(m))$ for $s\in[q+1]$.
 		We define the set $A_{ij}= A_{ij}(u)$ as follows:
		Let $A_{ij}:=P^i_{t_i}(m)\cap P^j_{t_j}(m)-\{(\mathbf{x},m)\}$ for distinct $i, j\in[q+1]$.

 		For $i \in [q+1]$, we define the sets $B_i = B_i(u)$, $C_i= C_i(u)$ and $R = R(u)$ by 
		\newline
		\noindent
		$B_i:=P^i_{t_i}(m)-\bigcup_{j\neq i}P^j_{t_j}(m)$,
		$C_i:=C(i,t_i,r^i_{t_i}(m))-P^i_{t_i}(m)$,  and 
		$R:=F(m)-\bigcup_{j=1}^{q+1}P^j_{t_j}(m)$.
		
		We first give the cardinalities of these sets and also show that they are mutually disjoint. 
 		 		
 		\begin{pro}\label{abc}
 			Let  $A_{ij}$, $B_i$, $C_i$, and $R$ be the sets as defined above.
 			These sets satisfy the following properties:
			
 			\begin{enumerate}
			        \item The sets $A_{ij}$, $B_h$, $C_p$, and $R$ are mutually disjoint for $i, j, h, p \in [q+1]$ and $i \neq j$; 
			         \item $|A_{ij}|=q-1$ for $i,j \in[q+1]$ and $i\neq j$;
 				\item $|B_i|=q-1$ for $i\in [q+1]$;
				\item  $|C_i|=(n-1)q^2$ for $ i \in [q+1]$;
 				\item $|R|=\frac{q(q-1)^2}{2}$ for $i\in[q+1]$.
 				
 			\end{enumerate}
 		\end{pro}
		\begin{proof}
		As $A_{ij}, B_h, R$ are all subsets of $F(m)$ and $C_p \cap F(m) = \emptyset$ for $i, j, h, p \in [q+1]$, $i \neq j$ so $C_p$ is disjoint from 
		$B_h$, $R$ and $A_{ij}$.  For distinct $i, j \in [q+1]$,  by Lemma \ref{inter}, we have $$ C(i,t_i, r^i_{t_i}(m))\cap C(j,t_j, r^j_{t_j}(m))\subset F(m).$$
		So $C_i \cap C_j = \emptyset$. Clearly the $B_i $ and $B_j$  are disjoint if $i \neq j$, and $A_{ij} \cap B_{h} = \emptyset$ if $i \neq j$. It is also clear that the set $R$ is disjoint 
		from the sets $A_{ij}$ and $B_h$.Let $i_1, i_2, j_1, j_2 \in [q+1]$ such that $i_1 \neq j_1$, $i_2 \neq j_2$ and $\{i_1, j_1\}\neq \{i_2, j_2\}$. Then 
		$A_{i_1, j_1} \cap A_{i_2, j_2}$ is contained in the intersection of at least three planes all containing $(\mathbf{x}, m)$. This means that  
		$A_{i_1, j_1} \cap A_{i_2, j_2} = \emptyset$, by Theorem \ref{cut}.
		These show that all the sets in (i) are disjoint.
		 
		Now we will determine the order of the sets $A_{ij}, B_i, C_i$ and $R$.
		For $i \neq j \in [q+1]$ the set $A_{ij} = P^i_{t_i}(m) \cap P^j_{t_j}(m)\setminus u$, and hence $|A_{ij}| = q-1$ by Theorem \ref{cut}.
		It follows that $|B_i|=|P^i_{t_i}|-\sum_{j\neq i}|A_{ij}|-1= q^2 -q(q-1) -1 = q-1$ for $i \in [q+1]$, and thus 
		$$|R|=q^3-\sum_i|B_i|-\sum_{i=1}^{q+1}\sum_{j=i+1}^{q+1}|A_{ij}| -|\{u\}|=q^3 - (q+1)(q-1) - \frac{(q+1)q}{2}(q-1) - 1=\frac{q(q-1)^2}{2}.$$
		Note that $|C_i|=|C(i,t_i,r^i_{t_i}(m))|-|P^i_{t_i}(m)|=(n-1)q^2$ for $i \in [q+1]$. This shows the proposition.
		\end{proof}
		
		As $F(m)$ induces a clique and $A_{ij}, B_h, R$ are subsets of $F(m)$, any two distinct vertices in any pair of these sets are adjacent. 
		
		Now we determine the number of neighbours of $w \in C_i$ where $i \in [q+1]$.
		
		\begin{lem} \label{abc2} With the above notation, let $w \in C_i$ where $i \in [q+1]$. Then $w$ has exactly 
		\begin{enumerate}
		\item $q(q-1)$ neighbours in $C_j$ for $j \in [q+1]-\{i\}$,
		\item $q-1$ neighbours in $A_{hj}$ if $i \in \{h, j\}$ and $0$ otherwise,
		\item $q-1$ neighbours in $B_j$ if $i=j$ and $0$ otherwise,
		\item $0$ neighbours in $R$. 
 				
 		\end{enumerate}
		\end{lem}
		\begin{proof}
		Note that $A_{ij}$, $B_i$ and $C_i$ are subsets of $C(i, t_i, r^i_{t_i}(m))$ for $i, j \in [q+1]$ and $i \neq j$. This shows (iii) and (ii) for $i \in \{h, j\}$.
		Assume that $w = (\mathbf{y}, p)$.  By Lemma \ref{plane}, $N(w) \cap F(m)$ is the plane $P^i_{t_i}(m)$, and hence $w$ has no neighbours in $A_{jh}$ if $ i \not \in \{j,, h\}$ by Theorem \ref{cut}. 
		This finishes the proof of (ii). It also follows that $w$ has no neighbours in $R$, showing (iv). 
		To show (i), let $ j \in [q+1] \setminus \{i\}$. We have $r^i_{t_i}(p) = r^i_{t_i}(m)$, and thus $r^j_{t_j}(p) \neq r^j_{t_j}(m)$ by the definition of  $TLS(q, n)$ and $\mathcal{O}$.
		We have $N(w) \cap F(m)= P^i_{t_i}(m)$, by Lemma \ref{plane}, of which exactly $q$ are in $C^j_{t_j}(m)$, namely the elements of $P^i_{t_i}(m) \cap P^j_{t_j}(m)$.
		As $w$ has $q^2$ neighbours in $C(j, t_j, r^j_{t_j})$ by Lemma \ref{xc}, it follows that $w$ has exactly $q^2-q$ neighbours in $C_j$.
		This finishes the proof of the lemma.
		
		\end{proof}
		
		Proposition \ref{abc} and Lemma \ref{abc2} implies that there are four kinds of vertices in the local graph $\Delta(u)$.
		To make this precise:
		\begin{pro}\label{lambda} With the above notation,
 		let $a\in A_{ij}$, $b\in B_i$, $c\in C_i$, and $v\in R$ for $i, j \in [q+1].$ Then we have:
 		\begin{enumerate}
 			\item $\lambda(u,a)=|F(m)-\{u,a\}|+|C_i|+|C_j|=q^3-2+2(n-1)q^2$;
 			\item $\lambda(u,b)=|F(m)-\{u,b\}|+|C_i|=q^3-2+(n-1)q^2$;
 			\item $\lambda(u,c)=\sum_{t\neq i}|N(c)\cap C_t|+|C(i,t_i,r^i_{t_i}(m))-\{u,c\}|=q(q^2-q)+nq^2-2=q^3-2 + (n-1)q^2$;
 			\item $\lambda(u,v)=|F(m)-\{u,v\}|=q^3-2$.
 		\end{enumerate}
		\end{pro}
		This gives immediately the following:
		\begin{pro}\label{equit}
		With the above notation,
 		let  $\mathcal{A} := \bigcup_i\bigcup_{j\neq i}A_{ij},  \mathcal{B} = \bigcup_i B_i,$ and $ \mathcal{C}:= \bigcup_i C_i$. 
		Then the partition $\{ \mathcal{A}, \mathcal{B}, \mathcal{C}, R\}$ is an equitable partition of the local graph $\Delta(u)$ with quotient matrix 
 		$$Q:= \begin{pmatrix}
 			\binom{q+1}{2}(q-1)-1 & (q+1)(q-1) & 2(n-1)q^2 & \frac{q(q-1)^2}{2} \\
 			\binom{q+1}{2}(q-1) & (q+1)(q-1)-1 & (n-1)q^2 & \frac{q(q-1)^2}{2} \\
 			q(q-1) & q-1 & (n-1)q^2+q(q^2-q)-1& 0\\
 			\binom{q+1}{2}(q-1) & (q+1)(q-1) & 0 & \frac{q(q-1)^2}{2}-1
 			 \end{pmatrix}.$$
                 \end{pro}
                 
                 As a consequence of the above two propositions we have the following result.
                 \begin{pro}\label{wer}
 			Each twisted Latin Square graph $TLS(q,n)$ is $(\alpha,\beta)$-weakly-edge-regular, where $\alpha+1 =(n-1)q^2 + q^3-2$ and $\beta=(2-q^3) - q^2(n-1)(q^2-q+1)$.
 		\end{pro}
 		
 		\begin{proof}
 			It can be directly confirmed by calculation, by using Propositions \ref{lambda} and \ref{equit}.
 		\end{proof}
		
		\begin{pro}\label{scer}
 			Each twisted Latin Square graph
 			$TLS(q,n)$ is $(\mu,\gamma)$-strongly co-edge-regular, where $\mu=(q+1)q^2$ and $\gamma=\mu(q^3-2+(n-1)q^2)$.
 		\end{pro}
		 	\begin{proof}
			
 		        Let $u=(\mathbf{x},m)$ be a vertex of the graph $TLS(q,n)$, where $\mathbf{x} \in\mathbb{F}^3_q$ and $ m \in [n^2]$.
		        We use the above notations $A_{ij}$, and so on. 
 			Let $w=(\mathbf{y},m')$ be a non-adjacent vertex to $u$. Then $m' \neq m$.
 			We have to consider two cases for $w$: either $N(w) \cap F(m) = \emptyset$ or $N(w) \cap F(m) \neq \emptyset$.
			In the first case, we have $N(w)\cap N(u)=N(w)\cap(\cup_{i=1}^{q+1}C_i)$.
 			This implies $|N(w)\cap N(u)|=(q+1)q^2$ and $$\sum_{v\in N(w)\cap N(u)}\lambda(u,v)= (q+1)q^2(q^3-2+(n-1)q^2)$$ by Lemma \ref{xc} and Proposition \ref{lambda}.
			
			If $N(w) \cap F(m) \neq \emptyset$, then, by Lemma \ref{plane}, there exists $s \in [q+1]$ and $ t \in [q]$ such that 
			$N(w) \cap F(m) = P^s_t(m)$. 
			We obtain that $|A_{ij} \cap P^s_t(m)| = 1$ for $i, j \in [q+1]-\{s\}$, and hence $|B_i \cap P^t_s(m)| = 1$ for $i \in [q+1] - \{s\}$.
			So that means $|\mathcal{A} \cap P^s_t(m)|= \frac{q(q-1)}{2}$ and $|\mathcal{B} \cap P^s_t(m)| = q$. 
			We obtain that $|R \cap P^s_t(m)| = q^2 - \frac{q(q-1)}{2}- q =  \frac{q(q-1)}{2}$.
			Also this implies that $w$ has exactly $q^2 - q$ neighbours in $C_i$ if $i \in [q+1]- \{s\}$ and exactly $q^2$ neighbours in $C_s$, by Lemma \ref{xc} and 
			Proposition \ref{lambda}. As $w$ has exactly $q^2$ neighbours in $F(m)$, we see that 
			 $$|N(w)\cap N(u)|=|N(w) \cap F(m)| + \sum_{i=1}^{q+1}  |N(w) \cap C_i| = q^2 + q^2 + q(q^2 -q) = (q+1)q^2.$$
			Now we have   
			\begin{align*} \sum_{v\in N(w)\cap N(u)}\lambda(u,v) =&\sum_{v\in N(w) \cap \mathcal{A}} \lambda(u,v) + \sum_{v\in N(w) \cap \mathcal{B}} \lambda(u,v) +
			\sum_{v\in N(w) \cap \mathcal{C}} \lambda(u,v)+ \sum_{v\in N(w) \cap R} \lambda(u,v)\\
			=&\frac{q(q-1)}{2}(q^3-2+2(n-1)q^2) + q(q^3-2+(n-1)q^2)+\\
			&(q(q^2-q) + q^2)(q^3-2+(n-1)q^2)  + \frac{q(q-1)}{2}(q^3-2)\\
			 =& (q+1)q^2(q^3-2+(n-1)q^2).
			\end{align*}
			
			This shows the proposition.
		         \end{proof}
		         
		         \begin{cor}
 			Each twisted Latin Square graph
 			$TLS(q,n)$ is a co-edge-regular graph of level $3$ with spectrum $$\{(q^3-1+q^2(q+1)(n-1))^{1},(q^2(n-1)-1)^{(qn-1)(q+1)},(-1)^{(q-1)q^2n^2},(-q^2-1)^{(qn-1)(n-1)q}\}.$$
 		\end{cor}
 		\begin{proof}
 			By Proposition \ref{lambda} and Proposition \ref{scer}, 	$TLS(q,n)$ is a co-edge-regular graph of grade $3$.
 			
 			By Lemma \ref{reg} and Proposition \ref{wer}, each twisted Latin Square graph $TLS(q,n)$ is  $k$-regular where 
			$k=q^3-1+(q+1)(n-1)q^2$, and weakly-edge-regular with $\alpha=(n-1)q^2 + q^3-3$ and $\beta=(2-q^3) - q^2(n-1)(q^2-q+1)$.
 			By Proposition \ref{scer}, $TLS(q,n)$ is strongly co-edge-regular  with $\mu=(q+1)q^2$ and $\gamma=\mu(q^2-q)q+nq^2-2$. 
 			Note that $TLS(q,n)$ is not a strongly regular graph  as there are three different valencies in any local graph of $TLS(q,n)$.
 			Therefore, $TLS(q,n)$ has exactly four distinct eigenvalues $k=\theta_0>\theta_1>\theta_2>\theta_3$ by Theorem \ref{4ev}.
 			
 			It follows from Lemma \ref{ev} that $\theta_1=q^2(n-1)-1$. By Theorem \ref{4ev} and Equation (\ref{1}),  we obtain
			$\alpha-\mu=f_1=\theta_1+\theta_2+\theta_3$ and $\mu(\alpha-1)+k-\beta-\gamma=f_2=\theta_1\theta_2+\theta_1\theta_3+\theta_2\theta_3$. 
			Therefore, $\theta_2=-1$ and $\theta_3=-q^2-1$. The multiplicities can be determined by using $Tr(A)=0$, $Tr(A^2)=q^3n^2k$, and that the sum of multiplicities 
			is equal to $q^3n^2=|V(TLS(q,n))|$.
 		\end{proof}
		
		As a consequence of the above corollary and Theorem \ref{proj}, we have the following result which provides a proof of Theorem \ref{intTLS}.
		\begin{thm}\label{TLS} 
			For a prime power $q$, there exist infinitely many integers $n$ such that a twisted Latin Square graph
 			$TLS(q,n)$ exists. The graph $TLS(q,n)$ is
 			a co-edge-regular graph with level $3$ and
 			 cospectral with the $q$-clique extension of each Latin Square graph $LS_{q+1}(qn)$.
		\end{thm}
		
		It worth mentioning that $TLS(q,n)$ is not switching equivalent to the $q$-clique extension of $LS_{q+1}(qn)$
		if $n\geq 6$.
		
		The smallest example of a twisted Latin Square graph is $TLS(2,2)$, which has $32$ vertices. Van Dam and Spence \cite{vs98} provided tables of connected regular graphs with 
		four distinct eigenvalues and at most $30$ vertices. 
		
		In particular, we can also consider cliques-extensions of twisted Latin Square graphs.
		As a consequence, for any positive integer $s\geq 2$, there exist infinitely many co-edge-regular graphs that are non-isomorphic but cospectral with the $s$-clique extension of certain Latin Square graphs
		
		\begin{thm}\label{TLScor}
			Let $q$ be a prime power. For any positive integer $s\geq 2$, such that $q$ is a factor of $s$. If $q\neq s$, then there exist infinitely many integers $n$ for which a co-edge-regular graph with level $4$ exists that is cospectral with the $s$-clique extension of each Latin Square graph $LS_{q+1}(qn)$.
		\end{thm}
		\begin{proof}		
		By Theorem \ref{TLS}, there exists infinitely many integers $n$ such that a twisted Latin Square graph
		$TLS(q,n)$ exists.
		Then the $\frac{s}{q}$-clique extension of $TLS(q,n)$ is a co-edge-regular graph of grade at least $3$ with the same spectrum as the $s$-clique extension of $LS_{q+1}(qn)$. Moreover, if $s\neq q$, then  the $\frac{s}{q}$-clique extension of $TLS(q,n)$ is a co-edge-regular graph of grade $4$.
		\end{proof}
		
		We conclude with some open problems:
		\begin{problem}
		\begin{enumerate}
		\item Let $G$ be a connected co-edge-regular graph with level $t$ and with exactly $d+1$ distinct eigenvalues. Is $t$ bounded by a function of $d$?
		Even for $d =3$ we do not know it.
		\item Is it possible to construct infinite families of cospectral graphs of certain clique extensions of certain Steiner graphs in a similar manner as in Section 5?
		\item Is it possible to generalize Theorem \ref{N51} to the class of co-edge-regular graphs of level 2 with exactly four distinct eigenvalues?
		\end{enumerate}
		\end{problem}
		\section*{Acknowledgements}
		J.H. Koolen is partially supported by the National Key R. and D. Program of China (No. 2020YFA0713100), the National Natural Science Foundation of China 
		(No. 12071454, No. 12471335), and the Anhui Initiative in Quantum Information Technologies (No. AHY150000).  We thank Dr. Qianqian Yang for her helpful comments.

 		\bibliographystyle{plain}
                 \bibliography{GK2024}
		\end{document}